\documentclass{amsart}
\usepackage{amscd}
\usepackage{color}
\usepackage[english]{babel}
\usepackage{amsmath}
\usepackage{amsfonts}
\usepackage{amsthm}
\usepackage{amssymb}
\usepackage{csquotes}
\usepackage{graphicx}
\usepackage{epstopdf}

\usepackage{bm}
\usepackage[all,cmtip]{xy}
\usepackage{hyperref}
\usepackage{enumerate}
\usepackage{enumitem}
\usepackage{mathrsfs}
\usepackage{vmargin}
\usepackage{verbatim}
\usepackage{cleveref}
\usepackage{stackrel}
\usepackage{epigraph}
\usepackage{chngcntr}
\usepackage{etoolbox}
\newtheorem{theo}{Theorem}[section]

\newtheorem{prop}[theo]{Proposition}
\newtheorem{lemma}[theo]{Lemma}
\newtheorem{claim}[theo]{Claim}

\theoremstyle{definition}
\newtheorem{defi}[theo]{Definition}
\theoremstyle{remark}

\newcommand{\BC}{{\mathbb{C}}}

\newcommand{\BM}{{\mathbb{M}}}
\newcommand{\BN}{{\mathbb{N}}}

\newcommand{\BQ}{{\mathbb{Q}}}

\newcommand{\BZ}{{\mathbb{Z}}}

\newcommand{\CC}{{\mathcal{C}}}

\newcommand{\CF}{{\mathcal{F}}}

\newcommand{\CM}{{\mathcal{M}}}

\newcommand{\CR}{{\mathcal{R}}}

\newcommand{\CT}{{\mathcal{T}}}

\newcommand{\Hom}{\mathop{\rm Hom}\nolimits}

\newcommand{\Colim}{{\rm Colim}}

\newcommand{\ra}{\rightarrow}

\newcommand{\rap}{\stackrel{+}{\rightarrow}}

\newcommand{\Sch}{\mathsf{Sch}}
\newcommand{\Sm}{\mathsf{Sm}}

\newcommand{\DA}{\mathop{\mathbf{DA}}\nolimits}

\newcommand{\MM}{\mathop{\mathbf{MM}}\nolimits}

\newcommand{\id}{{\rm id}}

\newcommand{\Gr}{{\rm Gr}}
\newcommand{\image}{\mathop{{\rm Im}}\nolimits}

\newcommand{\Coker}{\mathop{\rm Coker}\nolimits}
\newcommand{\Ker}{\mathop{\rm Ker}\nolimits}

\newcommand{\Jac}{\mathop{\rm Jac}\nolimits}

\newcommand{\DG}{\mathcal{DG}}

\newcommand{\Cone}{{\rm Cone}}

\newcommand{\et}{\mathrm{\acute{e}t}}

\newcommand{\coh}{\mathrm{coh}}

\newcommand{\red}{\mathrm{red}}
\newcommand{\sm}{\mathrm{sm}}

\newcommand{\perf}{{\rm perf}}

\date{\today}
\title{Constructible $1$-motives and exactness}
\author{Simon Pepin Lehalleur}
\thanks{The author thanks the Freie Universit\"at Berlin, where the research on this work was conducted. The author also acknowledges the support of the Einstein Foundation, through the Einstein visiting Fellowship 0419745104 ``Algebraic Entropy, Algebraic Cycles'' of Professor V. Srinivas.}
\hypersetup{
 pdfauthor={Simon Pepin Lehalleur},
 pdftitle={Constructible 1-motives and exactness},
 pdfkeywords={},
 pdfsubject={},
 pdfcreator={Emacs 25.3.50.2 (Org mode 9.1.2)}, 
 pdflang={English}}

\begin{document}

\newcommand{\TODO}{{\color{red} TODO }}
\newcommand{\REF}{{ \color{green} REF}}

\maketitle

\begin{abstract}
We prove that arbitrary pullbacks, as well as Betti and \'{e}tale realisation functors, are t-exact for the constructible motivic t-structure on the category of cohomological 1-motives over a base scheme.
\end{abstract}

\tableofcontents

\section{Introduction}

This paper takes place in the context of triangulated categories of mixed motivic sheaves in the sense of Morel-Voevodsky, and is a follow-up to \cite{phd_paper}. Let $S$ be a finite dimensional noetherian excellent scheme $S$. In \cite{phd_paper}, we constructed a candidate for the motivic t-structure on the triangulated category $\DA^{1}(S)$ of cohomological $1$-motives with rational coefficients over a noetherian finite-dimensional scheme $S$. In this note, we prove (under a mild hypothesis on $S$) that pullbacks along arbitrary morphisms are t-exact (\ref{main_theorem} \ref{pullback}), and that Betti and \'{e}tale realisation functors are t-exact if the target categories of realisation functors are equipped with their standard t-structure  (Theorem \ref{main_theorem} \ref{pullback} \ref{Realisation_Etale} \ref{Realisation_Betti}). Along the way, we prove that the t-structure on $\DA^{1}(S)$ restricts to the category $\DA^{1}_{c}(S)$ of compact cohomological $1$-motives (Theorem \ref{main_theorem} \ref{compactness}). 

Let $\MM^{1}_{c}(S)$ be the category of constructible cohomological $1$-motives, that is, the heart of the restricted t-structure on $\DA^{1}_{c}(S)$. In the course of the proof of the results above, we show that, for any $M\in\MM^1_c(S)$, there exists a stratification of the base $S$ by regular subschemes such that the restrictions of $M$ to strata are Deligne $1$-motives (Theorem \ref{main_theorem} \ref{structure}).

Theorem \ref{main_theorem} \ref{compactness} has been obtained previously by V. Vaish in \cite{vaish_gluing}. Vaish's approach relies on an elegant combination of the gluing procedure for t-structures of \cite{BBD} and the \enquote{weight truncation} t-structures of \cite{Morel}. He first gives an alternative construction of the functor $\omega^1:\DA^\coh_c(S)\ra \DA^1_c(S)$ (see Definition \ref{def:omega1} and Theorem \ref{thm:finiteness_omega1}) by gluing the analoguous functors $\omega^1:\DA^\coh_c(k(s))\ra \DA^1_c(k(s))$ for all points $s\in S$ (which exist by \cite{bvk}), and then uses  gluing data of the form $(j_!\dashv j^*\dashv \omega^1 j_*,i^*\dashv i_*\dashv \omega^1 i^!)$ for $j:U\ra S\leftarrow Z:i$ complementary open and closed immersions to glue together the t-structures on the $\DA^1_c(k(s))$ (which exist by \cite{orgogozo}). It is not clear to us how to prove the results of this note using the approach of \cite{vaish_gluing}; we plan to come back to this point and to combine the strengths of our approaches in a future work.

\section*{Acknowledgements}

This work was started during my PhD thesis under the supervision of Joseph Ayoub at the University of Z\"urich. I would like to thank him again for generously sharing his ideas and expertise on mixed motives. I have also had useful exchanges on the topic with Annette Huber, Florian Ivorra, Vaibhav Vaish and Johann Haas.

\section{Conventions}
All schemes are assumed to be finite dimensional, noetherian and excellent. Unless specified, smooth morphisms are assumed to be separated of finite type. The notation $\Sm/S$ denotes the category of all smooth $S$-schemes considered as a site with the \'etale topology.

\begin{defi}
We say that a scheme $S$ allows resolution of singularities by alterations if for any separated $S$-scheme $X$ of finite type and any nowhere dense closed subset $Z\subset X$, there is a projective alteration $g:X''\ra X$ with $X'$ regular and such that $g^{-1}(Z)$ is a strict normal crossing divisor.  
\end{defi}
 The best result available in this direction is due to Temkin \cite[Theorem 1.2.4]{Temkin_distillation}: any $S$ which is of finite type over a quasi-excellent scheme of dimension $\leq 3$ allows resolution of singularities by alterations.

 \section{Recollections on $1$-motives}

 For the comfort of the reader, we review some definitions and results from \cite{phd_paper}. Let $S$ be a scheme. The category $\DA(S):=\DA^\et(S,\BQ)$ is the triangulated category of rational \'etale motives coming from the stable homotopical $2$-functor $\DA^\et(-,\BQ)$ considered in \cite[\S 3]{Ayoub_Etale}. 

 \begin{defi}
\label{def:subcats}
The category $\DA^{\coh}(S)$ of \emph{cohomological motives} is the localising subcategory of $\DA(S)$ generated by
\[
\{f_*\BQ_X|\mbox{ $f:X\rightarrow S$ proper morphism}\}.
\] 
The category $\DA^1(S)$ of \emph{cohomological $1$-motives} is the localising subcategory of $\DA(S)$ generated by
\[
\{f_*\BQ_X|\mbox{ $f:X\rightarrow S$ proper morphism of relative dimension $\leq 1$}\}.
\]
\end{defi}

\begin{defi}\label{def:omega1}
The full embedding $\iota^1:\DA^1(S)\hookrightarrow \DA^\coh(S)$ preserves small sums, thus by Neeman's version of Brown representability for compactly generated triangulated categories  (see e.g. \cite[Theorem 8.3.3]{Neeman_book}), admits a right adjoint $\omega^1:\DA^\coh(S)\ra \DA^1(S)$.
\end{defi}

One of the main results of \cite{phd_paper} (which was reproved later by Vaish in \cite{vaish_gluing} with a different method) is the following.

\begin{theo}\label{thm:finiteness_omega1}
Let $S$ be a noetherian finite-dimensional excellent scheme. Assume that $S$ allows resolution of singularities by alterations. Then the functor $\omega^1:\DA^\coh(S)\ra \DA^1(S)$ preserves compact objects.
\end{theo}

We recall the definition of Deligne $1$-motives.

\begin{defi}\label{def:1-mot}
 Let $S$ be a scheme. A $2$-term complex of commutative $S$-group schemes:
  \[ M=[\stackrel{0}{L}\longrightarrow \stackrel{-1}{G}]
  \]
is called a \emph{Deligne $1$-motive} over $S$ if $L$ is a lattice and $G$ is a semi-abelian scheme. We denote by $\CM_1(S)$ the category of Deligne $1$-motives with rational coefficients (i.e., the idempotent completion of the category whose morphisms groups are morphism groups of Deligne $1$-motives tensored with $\BQ$).
\end{defi}

Recall that sets of isomorphism classes of compact objects freely generate t-structures in compactly generated triangulated categories \cite[Lemme 2.1.69, Proposition~2.1.70]{Ayoub_thesis}.

\begin{defi}\label{main_def_1}
The \emph{motivic t-structure} $t^1_{\MM}(S)$ on $\DA^1(S)$ is the t-structure generated by the family
  \[
\DG_S=\left\{e_\sharp\Sigma^\infty(\BM)|\ e:U\ra S\text{ \'etale },\ \BM\in \CM_1(U)\right\}.
\]
of compact objects.
\end{defi}

Let $(\CT,\CT_{\geq 0},\CT_{<0})$ be a triangulated category with a t-structure, written with the homological convention. In \cite{phd_paper} we used the terminology \enquote{t-positive} for objects in $\CT_{\geq 0}$ and \enquote{t-negative} for objects in $\CT_{\leq 0}$; we adopt here the more correct english usage of \enquote{t-non-negative} for objects in $\CT_{\geq 0}$ and \enquote{t-non-positive} for objects in $\CT_{\leq 0}$.

The main properties of $t^1_{\MM}(S)$ from \cite[\S 4]{phd_paper} which we will use are the following.
\begin{itemize}
\item Elementary exactness properties \cite[Proposition 4.14]{phd_paper}.
\item Compact objects are bounded for $t^1_{\MM}(S)$ \cite[Corollary 4.28]{phd_paper}.
\item There is a functor $\Sigma^\infty(-)(-1):\CM_1(S)\ra \MM^1(S)$ which is fully faithful when $S$ is regular \cite[Theorem 4.21, Theorem 4.30]{phd_paper}. Furthermore, if $e:U\ra S$ is any \'{e}tale morphism and $\BM\in\CM_1(U)$, then $e_!\Sigma^\infty\BM(-1)\in \MM^1(S)$.
\end{itemize}

For the theory over an imperfect field, we have the following results which complements the treatment in \cite{phd_paper}.

\begin{prop}\label{prop:D1_imperfect}
Let $k$ be a field and $l/k$ be a purely inseparable field extension. Then the base change functor
  \[
\CM_1(k)\ra \CM_1(l)
\]
is an equivalence of categories.
\end{prop}
\begin{proof}
  Let us first prove that this functor is fully faithful. Even though $\CM_1(k)$ is defined as an idempotent completion, it is clearly enough to prove full faithfulness before passing to the idempotent completion. Let $\BM=[L\stackrel{u}{\ra} G]$ and $\BM'=[L'\stackrel{u'}{\ra} G']$ in $\CM_1(k)$. 

Since $\BQ$ is flat over $\BZ$, it is enough to show faithfulness for the functor $\CM_1(k,\BZ)\ra \CM_1(l,\BZ)$. The base change functor from group schemes over $k$ to group schemes over $l$ is faithful. This concludes the proof of faithfulness.

Since the category of lattices only depends on the small \'{e}tale site $k_{\et}$ and $k_{\et}\simeq l_{\et}$, we have $\CM_1(k)(L,L')\ra \CM_1(k^\perf)(L_{l},L'_{l})$. By \cite[Theorem 3.11]{Brion_iso}, $\Hom_{\underline{\CC}_k}(G,G')\simeq \Hom_{\underline{\CC}_l}(G_l,G'_l)$ with $\underline{\CC}_k$ the category of smooth commutative $k$-group up to isogeny in the sense of loc.cit. Since semi-abelian varieties over any field are divisible, this implies by \cite[Proposition 3.6]{Brion_iso} that
\[
\CM_1(k)(G,G'):= \Hom(G,G')\otimes\BQ\simeq \Hom_{\underline{\CC}_k}(G,G')\simeq \Hom_{\underline{\CC}_l}(G_l,G'_l)\simeq \CM_1(l)(G,G').
\]
We can now prove fullness. Let $g=f\otimes \frac{1}{n}\in\CM_1(l)(\BM_l,\BM'_{l})$ with $f=(f^L,f^G)\in \CM_1(k,\BZ)(\BM_{l},\BM'_{l'})$. By the previous paragraph, there exist preimages $f_0^L:L\ra L'$ and $f_0^G:G\ra G'$ of $f^L$, $f^G$. The pair $(f_0^L,f_0^G)$ is a morphism of complexes if and only if $u'\circ f_0^L =f_0^G\circ u:L\ra G'$. Because the base change for group schemes from $k$ to $l$ is faithful, we can check this over $l$, where it follows from the fact that $f$ is a morphism. This concludes the proof of fullness.

We prove essential surjectivity. Let $\BM\in \CM_{1}(l)$. By full-faithfulness, we can ignore the idempotent completion and assume that $\BM=[L\stackrel{u}{\ra}G]$. By etale descent and semi-simplicity of lattices up to isogeny, we can assume furthermore that $L\simeq \BZ^r$ is split. By a finite generation argument (note that $L$ is not finitely presented as a scheme, but it is as a group scheme), we see that $\BM$ comes from a finitely generated (hence finite since it is purely inseparable) subextension of $l$. Hence we assume that $l$ is finite over $k$, say $l^{q}\subset k$ with $q=p^{N}$ is a large enough power of $p$. There is a lattice $L_0$ over $k$ such that $L\simeq (L_0)_l$ as group schemes. By \cite[Theorem 3.11]{Brion_iso} (again combined with \cite[Proposition 3.6]{Brion_iso} and divisibility of semi-abelian varieties), there exists a semi-abelian variety $G_{0}$ over $k$ and an isogeny $\lambda:G\ra (G_{0})_l$. We thus get a morphism $\lambda u:\BZ^r\ra (G_0)_l$. By \cite[Exp. VIIA \S 4.3]{SGA3_1}, we see that $[q]\lambda u$ factors through a morphism $L_0\ra G_0$, which makes \([L_0\ra G_0]\in \CM_1(k)\) into a pre-image of \(\BM\). This concludes the proof.
\end{proof}

\begin{prop}\label{prop:field}
Over a field $k$, the t-structure restricts to compact objects and the functor $\Sigma^\infty(-)(-1):D^b(\CM_1(k))\ra \DA^1_c(k)$ is an equivalence of t-categories, so that $\Sigma^\infty(-)(-1):\CM_1(k)\simeq \MM^1_c(k)$.
\end{prop}
\begin{proof}
This follows from \cite[Proposition 4.20]{phd_paper} combined with Proposition \ref{prop:D1_imperfect}.
\end{proof}

In the same vein, here is a result implicit in \cite{phd_paper} which we make explicit for later reference.

\begin{lemma}\label{lemma:radicial}
  Let $f:T\ra S$ be a finite surjective radicial morphism. Then $f^*:\DA^1(S)\ra \DA^1(T)$ is an equivalence of t-categories.
\end{lemma}
\begin{proof}
For such a morphism, $f^*\simeq f^!:\DA_{(c)}(S)\ra\DA_{(c)}(T)$ is an equivalence by \cite[Corollaire 2.1.164]{Ayoub_thesis}. Since $f$ is finite, the functor $f_*$ sends $\DA^1(S)$ to $\DA^1(T)$, so that $f^*$ induces an equivalence between $\DA^1(S)$ and $\DA^1(T)$. Finally, the t-exactness follows from $f^*\simeq f^!$ and \cite[Proposition 4.14]{phd_paper}.
\end{proof}

\section{Constructible $1$-motives}
\label{sec:org484b0d5}

Here is the main theorem of this paper.

\begin{theo}\label{main_theorem}
Let $S$ be an excellent scheme allowing resolution of singularities by alterations. Then 
\begin{enumerate}[label={\upshape(\roman*)}]
\item\label{compactness} The t-structure $t^1_{\MM}(S)$ restricts to the subcategory $\DA^1_c(S)$ of constructible $1$-motives. Denote its heart by $\MM^1_c(S)$.
\item\label{structure} Let $M$ be in $\DA^1_c(S)$. Then $M$ is in $\MM^1(S)$ if and only if there exists a locally closed stratification $(S_{\alpha})$ of $S_{\red}$ such that for all $\alpha$, if $i_{\alpha}:S_\alpha\ra S$ is the immersion, we have 
\[
  i^*_{\alpha} M\simeq \Sigma^\infty \BM_\alpha(-1)
\]
with $\BM_\alpha$ a Deligne $1$-motive on $S_{\alpha}$. Moreover, we can assume the $S_{\alpha}$ to be regular.
\item \label{pullback} Let $f:T\ra S$ be a morphism. Then the functors $f^*:\DA^1_c(S)\ra \DA^1_c(T)$ and $f^*:\DA_{1,c}(S)\ra \DA^1_c(S)$ are t-exact (with respect to the restricted t-structures from point $(i)$).
\item \label{Realisation_Etale} Let $\ell$ be a prime number invertible on $S$. Then the functor $R_\ell: \DA^1_c(S)\ra D^b_c(S_\et,\BQ_\ell)$ obtained by restricting the rational $\ell$-adic realisation functor from \cite[Definition 9.6]{Ayoub_Etale} is t-exact for the motivic t-structure of \ref{compactness} on the source and the standard t-structure on the target.
\item \label{Realisation_Betti} Assume that $S$ is a $k$-scheme with $k$ a field of characteristic 0 admitting an embedding $\sigma:k\ra \BC$. Then the Betti realisation functor $R_{B,\sigma}: \DA^1_c(S)\ra D^b_c(S_\sigma(\BC),\BQ)$ is t-exact for the motivic t-structure of \ref{compactness} on the source and the standard t-structure on the target.
\item \label{other_cats} Statements \ref{structure}-\ref{Realisation_Betti} also hold for homological $1$-motives $\DA_1(-)$ (resp. for $0$-motives $\DA^0(-)$) provided one replaces $\BM(-1)$ by $\BM$ (resp. by $\CF$ with $\CF$ a locally free sheaf of $\BQ$-vector spaces) in \ref{structure} (cf. \cite{phd_paper} for the relevant definitions for homological $1$-motives and $0$-motives).
\end{enumerate}
\end{theo}

\begin{proof}
  First, a word of caution about notation. Since we do not know yet that $t^1_{\MM}(S)$ restricts to compact objects, we refrain from using the notation $\MM^1_c(S)$ and always write $\MM^1(S)\cap \DA^1(S)_c$ for the compact objects in the heart.
  
\noindent  
\underline{First reductions:}
We first prove the \enquote{if} direction of Statement \ref{structure}. This is an immediate consequence of the following lemma.

\begin{lemma}\label{lemma:str_heart}
Let $S$ be a scheme as in the statement of Theorem \ref{main_theorem} and $M\in \DA^1_c(S)$. Assume that there exists a stratification $\{S_\alpha\}$ of $S$ so that we have 
\[
i^*_{\alpha} M\simeq\Sigma^\infty \BM_{\alpha}(-1)
\]
with $\BM_{\alpha}$ a Deligne $1$-motive on $S_{\alpha}$. Then $M$ is in $\MM^1(S)$. 
\end{lemma}
\begin{proof}
We prove this statement by induction on the dimension of $S$. The result clearly holds in dimension $0$. Let $S$ be a scheme of dimension $\leq d$, and $M\in \DA^1_c(S)$. Let $\{S_{\alpha}\}$ be a stratification of $S$ as in the statement. Write $U$ for the union of the open strata, and $Z$ for the complement, equipped with the reduced scheme structure (i.e. the union of all the other strata). Write $j:U\ra S$ for the open immersion and $i:Z\ra S$ for the complementary closed immersion. Then $Z$ is of dimension $<d$. We see that $i^*M$ satisfies the same hypothesis, with the restricted stratification (since the pullback of a Deligne $1$-motive is a Deligne $1$-motive, \cite[Corollary 2.15]{phd_paper}). By the induction hypothesis, the motive $i^*M$ is in $\MM^1(Z)$; by \cite[Proposition 4.14]{phd_paper}, we get $i_*i^*M\in\MM^1(S)$. Moreover, $j^*M$ is a Deligne $1$-motive. By \cite[Theorem 4.21]{phd_paper}, this implies that $j_!j^*M$ is in $\MM^1(S)$. By localisation, we have a distinguished triangle
\[
j_!j^*M\ra M\ra i_*i^*M\rap
\]
which shows that $M\in \MM^1(S)$ as required.
\end{proof}

We introduce the following claim which depends on an integer $d\in \BN$.

\begin{claim}
$(\mathrm{Str}_d)$: for any $S$ as in the theorem, of dimension $\leq d$, the t-structure $t^1_{\MM}(S)$ restricts to $\DA^1_c(S)$ and any $M\in \MM^1(S)\cap \DA^1_c(S)$, there exists a locally closed stratification $(S_{\alpha})$ of $S_{\red}$ such that for all $\alpha$, the scheme $S_{\alpha}$ is regular and, if $i_{\alpha}:S_\alpha\ra S$ is the natural immersion, we have 
\[
  i^*_{\alpha} M\simeq \Sigma^\infty\BM_\alpha(-1)
\]
with $\BM_\alpha$ a Deligne $1$-motive on $S_{\alpha}$.
\end{claim}

Clearly, the statements $(\mathrm{Str}_{d})$ for all $d$ are equivalent to the combination of \ref{compactness} and \ref{structure}. We will prove $(\mathrm{Str}_{d})$ by induction on $d$. Before we proceed to do so, let us show that \ref{compactness} and \ref{structure} imply all the other claims of Theorem \ref{main_theorem}.
 
We show \ref{Realisation_Etale} and \ref{Realisation_Betti} assuming \ref{compactness} and \ref{structure}. The argument is the same in both cases, so we only present the $\ell$-adic case. By \cite[Corollary 4.28]{phd_paper}, the motivic t-structure on $\DA^1_c(S)$ is bounded, so that to prove t-exactness it is enough to show that an object in the heart $\MM^1(S)\cap \DA^1_c(S)$ is sent to a constructible $\ell$-adic sheaf. Let $M\in \MM^1(S)\cap \DA^1_c(S)$. Assuming \ref{structure}, there exists a locally closed stratification $(S_{\alpha})$ of $S_{\red}$ such that for all $\alpha$, if $i_{\alpha}:S_\alpha\ra S$ is the natural immersion, we have 
\[
  i^*_{\alpha} M\simeq \Sigma^\infty \BM_\alpha(-1)
\]
with $\BM_\alpha$ a Deligne $1$-motive on $S_{\alpha}$. By gluing and exactness of pullbacks for $\ell$-adic sheaves, it is enough to show that for any $\alpha$, the object $i^*_\alpha R_\ell M\simeq R_\ell (\Sigma^\infty \BM_\alpha)(-1)$ is a constructible $\ell$-adic sheaf. This fact is established in the proof of \cite[Proposition 4.15]{phd_paper}.

We show \ref{pullback} assuming \ref{compactness} and \ref{structure}. This follows from the slightly more precise lemma, which will also be used in the induction below.

\begin{lemma}\label{lemma:str_pull}
  Assume that statement $(\mathrm{Str}_d)$ holds. Let  $f:T\ra S$ be a morphism of schemes with $\dim(S)\leq d$ and $\dim(T)\leq d$. By assumption, $t^{1}_{\MM}(S)$ and $t^{1}_{\MM}(T)$ restrict to the subcategories of compact objects. Then the functor $t^*:\DA^{1}_{c}(T)\ra \DA^{1}_{c}(S)$ is t-exact. 
\end{lemma}

\begin{proof}
By Lemma \ref{lemma:radicial}, we can assume that $S$ and $T$ are reduced. By \cite[Proposition 4.14]{phd_paper}, it is enough to show that $f^*:\DA^1_c(S)\ra \DA^1_c(T)$ is t-non-positive. Let $M\in \DA^1_c(S)_{\leq 0}$. We have to show that $f^*M$ is t-non-positive. By \cite[Corollary 4.28]{phd_paper}, the motive $M$ has finitely many non-zero homology objects, and thus can be obtained by finitely many extensions starting with non-positive shifts of objects in $\MM^1(S)\cap \DA^1_c(S)$. So it is enough to prove that $M\in \MM^1(S)\cap \DA^1_c(S)\Rightarrow f^*M$ is t-non-positive.

By $(\mathrm{Str}_d)$, there exists a stratification $\{S_\alpha\}$ of $S$ so that we have 
\[
i^*_{\alpha} M\simeq\Sigma^\infty \BM_{\alpha}(-1)
\]
with $\BM_{\alpha}$ a Deligne $1$-motive on $S_{\alpha}$. Consider the induced stratification $T_\alpha:=f^{-1}(S_\alpha)$ of $T$, with $i'_{\alpha}:T_\alpha\ra T$ and $f_\alpha:T_{\alpha}\ra S_{\alpha}$. The $T_{\alpha}$ are not necessarily regular, but this does not matter for the rest of the argument. By \cite[Corollary 2.14]{phd_paper}, we have $i'_{\alpha}f^*M\simeq f_{\alpha}^*\Sigma^\infty \BM_{\alpha}(-1)$ is a Deligne $1$-motive. By Lemma \ref{lemma:str_heart}, this implies that $f^*M$ is in $\MM^1(T)\cap \DA^1_c(S)$. This concludes the proof.
\end{proof}

\noindent
\underline{Reduction to the degeneration of Deligne $1$-motives:}

All that remains is to establish $(\mathrm{Str}_d)$ by induction on $d\geq 0$. 

Let us verify the case $d=0$. We can clearly assume $S$ to be reduced, and then by considering connected components, the result follows from Proposition~\ref{prop:field}.

We now assume $d\geq 1$ and that $(\mathrm{Str_{d-1}})$ holds. Let $S$ be a scheme of dimension $\leq d$. By Lemma \ref{lemma:radicial} we can assume that $S$ is reduced.

To prove that $t^{1}_{c}(S)$ restricts to compact objects, we have to prove that for any $M\in\DA^{1}_{c}(S)$, we have $\tau_{\geq 0}M\in\DA^{1}_{c}(S)$. By \cite[Corollary 4.28]{phd_paper}, $M$ is bounded for $t^1_{\MM}(S)$. By induction on the $t^{1}_{\MM}(S)$-amplitude of $M$, we see that it is enough to show that for a morphism $f:A\ra B$ with $A,B\in \MM^1(S)\cap \DA_c(S)$, the motive $\tau_{\geq 1} \Cone(f)\simeq \Ker_{\MM^1(S)}(f)$ is compact (or equivalently, that $\tau_{\leq 0}\Cone(f)\simeq \Coker_{\MM^1(S)}(f)$ is compact).

The idea is to describe $f$ over a dense open of $S$ in terms of Deligne $1$-motives and to try to degenerate it to the boundary and apply the induction hypothesis. Let us first describe $A$ and $B$ generically.

Let $\eta$ be the scheme of generic points of $S$, i.e., a disjoint union of spectra of fields. Let us show that $\eta^*A$ lies in $\MM^1(\eta)\cap \DA^{1}_{c}(S)$. The functor $\eta^*$ is t-non-negative by Proposition~\cite[Proposition 4.14]{phd_paper} (where no finite type hypothesis is required). Let us show that $\eta^*A$ is t-non-positive. We have to show that for any $P$ in a compact generating family of $t^1_{\MM}(\eta)$ and $n>0$, we have $\DA(\eta)(P[n],\eta^*A)\simeq 0$.  By \cite[Proposition 4.20]{phd_paper}, we can assume for instance that $P$ is of the form $\Sigma^\infty\BM'_{\eta}(-1)$ with $\BM'_{\eta}\in \CM_1(\eta)$.

By continuity results for Deligne $1$-motives \cite[Proposition A.10]{phd_paper} and $\DA_c(-)$, we can find an open set $\eta\in V\stackrel{j}{\ra} S$ such that there exist $Q= \CR \Sigma^\infty\widetilde{\BM}''(-1)$ for $\widetilde{\BM}''\in \CM_1(V)$ and $P\simeq (\eta/V)^*Q$. In particular, in both cases, $Q$ is t-non-negative. We can then use continuity for $\DA_c(-)$ to write
\begin{eqnarray*}  
  \DA(\eta)(P[n],\eta^*A) &\simeq & \DA(\eta)((\eta/V)^* Q[n],(\eta/V)^*j^*A)\\
                                    &\simeq & \Colim_{\eta\in W\subset V} \DA(W)((W/V)^*Q[n], (W/V)^* j^*A).
\end{eqnarray*}
For such an intermediate open $W$, we see from \cite[Proposition 4.14]{phd_paper} that $(W/V)^*Q$ is t-non-negative while $(W/V)^*j^*A$ is t-non-positive. This implies that every morphism group in the colimit vanishes, and completes the proof that $\eta^*A$ is in $\MM^1(\eta)\cap \DA^1_c(\eta)$.

Over $\eta$, which is a finite disjoint union of spectra of fields, thanks to Proposition \ref{prop:field}, we understand completely the structure of compact objects in $\MM^1(\eta)$. Namely, there exists a Deligne $1$-motive $\BM_\eta\in \CM_1(\eta)$ such that 
\[
\eta^*A\simeq \Sigma^\infty \BM_\eta(-1)
\]
The motive $\BM_\eta$ has three components $\Gr^W_i \BM_\eta$ for $i=-2,-1,0$. We can write $\Gr^W_{-1}\BM_\eta$ as a direct factor of the Jacobian of a smooth projective geometrically connected curve $C_{\eta}$ of genus $\geq 2$ with a rational point $\sigma_{\eta}$; that is, $\Gr^W_{-1}\BM_{\eta}=\image(\pi_{\eta}:\Jac(C_\eta)\otimes\BQ \ra \Jac(C_\eta)\otimes\BQ)$ with $\pi_\eta$ an idempotent. 

By continuity for Deligne $1$-motives and $\DA(-)$, we can find an open set $\eta\in U\stackrel{j}{\ra} S$ such that $j^*A\simeq \Sigma^\infty\widetilde{\BM}(-1)$ for $\widetilde{\BM}\in \CM_1(U)$. We can also assume that $U$ is regular, and that $\Gr^W_{-1}\BM$ is a direct factor of the Jacobian of a smooth projective curve $f:C\ra U$ with geometrically connected fibers which comes together with a section $\sigma:U\ra C$.

The same arguments apply verbatim for $B$, so that we can assume as well that $j^*B\simeq \Sigma^\infty\widetilde{\BN}(-1)$ for $\widetilde{\BN}\in \CM_1(U)$ with $\Gr^W_{-1}\widetilde{\BN}$ a direct factor of the Jacobian of a smooth projective curve $f:P\ra U$ with geometrically connected fibers which comes together with a section $\theta:U\ra P$.

The functor $\Sigma^{\infty}(-)(-1):\CM_1(U)\ra \MM^1(U)$ is fully faithful by \cite[Theorem 4.30]{phd_paper}, so that we can identify the morphism $j^*f:j^*A\ra j^* B$ modulo the isomorphisms above with a morphism $F:\widetilde{\BM}\ra \widetilde{\BN}$. Applying continuity for $\CM_1(-)$ and $\DA_c(-)$. Restricting $U$, we can also assume that the kernel $K$ and cokernel $Q$ of $F$ in the abelian category $\CM_1(\eta)$ extend to Deligne $1$-motives $\widetilde{K}$, $\widetilde{Q}$ over $U$. Consider the morphism $\Sigma^\infty F(-1)$. The cone of $\Cone(\eta^* \Sigma^\infty F(-1))\simeq \Cone(\eta^*f)$ fits into a distinguished triangle
\[
\eta^*\Sigma^\infty \widetilde{K}(-1)[1]\ra \eta^*\Cone(\Sigma^\infty \widetilde{F}(-1))\ra \eta*\Sigma^\infty \widetilde{Q}(-1)\rap.
\]
By continuity for $\DA(-)$, again by restricting $U$, we can assume there is a distiguished triangle
\[
\Sigma^\infty \widetilde{K}(-1)[1]\ra \Cone(\Sigma^\infty \widetilde{F}(-1))\ra \Sigma^\infty \widetilde{Q}(-1)\rap.
\]
By \cite[Theorem 4.21]{phd_paper}, we have $j_!\Sigma^\infty \widetilde{K}(-1)\in \MM^1(S)$ and $\ j_!\Sigma^\infty \widetilde{Q}(-1)\in \MM^1(S)$. This implies that the distinguished triangle
\[
j_!\Sigma^\infty \widetilde{K}(-1)[1] \ra j_!j^* \Cone(f) \ra j_!\Sigma^\infty\widetilde{Q}(-1)\rap
\]
is the truncation triangle of $j_!j^* \Cone(f)$ for $t^1_{MM}$, i.e., $H_1(j_!j^* \Cone(f))\simeq j_!\Sigma^\infty \widetilde{K}(-1)$ and $H_0(j_!j^* \Cone(f))\simeq j_!\Sigma^\infty \widetilde{Q}(-1)$.

Let us now prove that $i^*A$ and $i^*B$ lie in $\MM^1(Z)$; it is enough to do so for $i^*A$. By \cite[Proposition 4.14]{phd_paper}, since $A$ is t-non-negative, the motive $i^*A$ is t-non-negative. It remains to show it is t-non-positive. Applying $\omega^1 i^*$ to the colocalisation triangle, we get
\[
\omega^1 i^!A\ra i^*A\ra \omega^1 i^*j_*j^*A\rap.
\]
The functor $\omega^1 i^!$ is t-non-positive by \cite[Proposition 4.14]{phd_paper}, hence it is enough to show that $\omega^1 i^*j_* j^* A$ is t-non-positive. We have seen that $j^*A$ is of the form $\Sigma^\infty \widetilde{\BM}$. By Lemma~\ref{lemma:omega_chi} below, $\omega^1 i^*j_* j^* A$ is t-non-positive, and we conclude that $i^*A$ lies in the heart as claimed.

By \cite[Proposition 4.14]{phd_paper}, the motives $i_*i^*A$ and $i_*i^*B$ are also in the heart, and we see that the morphism of localisation triangles
\[
\xymatrix{
j_! \Sigma^\infty \widetilde{\BM}(-1) \ar[r] \ar[d]_{j_!\Sigma^\infty \widetilde{F}(-1)} & A \ar[r] \ar[d]_{f} & i_*i^*A \ar[d]_{i_*i^*f} \ar^{+}[r] & \\
j_! \Sigma^\infty \widetilde{\BM}'(-1) \ar[r] & B \ar[r] & i_*i^*B \ar^{+}[r] & \\
}
\]
is in fact a morphism of short exact sequences in $\MM^1(S)$, to which we can apply the Snake lemma and get a six term exact sequence
\begin{multline*}
  0\ra j_!\Sigma^\infty \widetilde{K}(-1)\ra H_1(\Cone(f))\ra H_1(\Cone(i_* i^*f))\\\ra j_!\Sigma^\infty \widetilde{Q}(-1) \ra 
 H_0(\Cone(f)) \ra H_0(\Cone(i_*i^* f))\ra 0.  
\end{multline*}
Again by Proposition \cite[Proposition 4.14]{phd_paper}, we have $H_1(\Cone(i_*i^*f))\simeq i_*H_1(\Cone(i^*f))$ and $H_0(\Cone(i_*i^*f))\simeq i_*H_0(\Cone(i^*f))$.
We can apply the induction hypothesis to $i^*A$ and the morphism $i^*f:i^*A\ra i^*B$ on the proper closed subset $Z$. We deduce that $H_1(\Cone(i^*f))$ and $H_0(\Cone(i^*f))$ are compact, and that there exists a locally closed stratification $(Z_\alpha)_{\alpha}$ of $Z$ such that for all $\alpha$, $Z_{\alpha}$ is regular and $i_{\alpha}^* A$ is of the form $\Sigma^\infty \BM_{\alpha}(-1)$ for some $\BM_{\alpha}(-1)$. 

Since $i_*$ preserves compact objects, $H_1(\Cone(i_*i^*f))$ and $H_0(\Cone(i_*i^*f))$ are compact. By adjunction and colocalisation, we have a sequence of isomorphisms
\begin{eqnarray*}
\DA(S)(i_*H_1(\Cone(i^*f)),  j_!\Sigma^\infty \widetilde{Q}(-1)) & \simeq & \DA(Z)(H_1(\Cone(i^*f)), i^!j_!\Sigma^\infty \widetilde{Q}(-1)) \\
& \simeq & \DA(Z)(H_1(\Cone(i^*f)), i^*j_*\Sigma^\infty \widetilde{Q}(-1)[-1]) \\
& \simeq & \DA(Z)(H_1(\Cone(i^*f)), \omega^1 i^*j_*\Sigma^\infty \widetilde{Q}(-1)[-1]).
\end{eqnarray*}
Again by Lemma~\ref{lemma:omega_chi} below, the motive $\omega^1 i^*j_*\Sigma^\infty \widetilde{Q}(-1)$ is t-non-positive. Since the motive $H_1(\Cone(i^*f))$ is t-non-negative, we deduce that the morphism group above vanishes. This shows that the six term exact sequence above splits into two short exact sequences,
\[
0\ra j_!\Sigma^\infty \widetilde{K}(-1)\ra H_1(\Cone(f))\ra H_1(\Cone(i_* i^*f))\ra 0
\]
and
\[
0\ra j_!\Sigma^\infty \widetilde{Q}(-1)\ra H_0(\Cone(f)) \ra H_0(\Cone(i_*i^* f))\ra 0.
\]
We have proved that the outer terms in both those sequences are compact, and we deduce that $H_1(\Cone(f))$ and $H_0(\Cone(f))$ are compact. This completes the proof of $(\Sch_d)$, modulo Lemma~\ref{lemma:omega_chi} below.

\begin{lemma}\label{lemma:omega_chi}
Assume that $(\Sch_{d-1})$ holds and that $S$ is a noetherian excellent finite dimensional scheme satisfying resolution of singularities by alterations. Let $j:U\ra S$ be an open immersion with $U_{\red}$ regular and $i:Z\ra S$ the complementary reduced closed immersion.  Let $\BM\in \CM_1(U)$ and $M=(\Sigma^\infty \BM)(-1)$. Assume moreover that the abelian scheme part of $\BM$ is a direct factor of the Jacobian scheme of a smooth projective curve with geometrically connected fibres. Then the motive $\omega^1 i^*j_* M$ is $t^1_{\MM}(Z)$-non-positive.
\end{lemma}

The end of the proof of \ref{main_theorem} is thus entirely devoted to the proof of Lemma~\ref{lemma:omega_chi}. 

\noindent
\underline{Reduction to normal crossings via alterations and stable curves:}

By localisation, we have a distinguished triangle
\[
\omega^1 i^!M\ra i^*M \ra \omega^1i^*j_*j^*M\rap
\]
and the motive $\omega^1i^!M$ is t-non-positive by \cite[Proposition 4.14]{phd_paper}. It is thus enough to prove that $\omega^1 i^*j_*M$ is t-non-positive.

Using Lemma \ref{lemma:radicial}, we can assume that $S$ is reduced. Let us show that we can in fact assume $S$ to be integral. Let $q:\widetilde{S}\ra S$ be the normalisation morphism. Since $U$ is assumed to be regular, $q$ is an isomorphism above $U$. Consider the diagram of schemes with cartesian squares
\[
\xymatrix{
  U\ar@{=}[d] \ar[r]_{\tilde{\jmath}} & \widetilde{S} \ar[d]_{q} & \widetilde{Z} \ar[l]^{\tilde{\imath}} \ar[d]_{q_Z} \\
  U \ar[r]_{j} & S & Z \ar[l]^{i}.
}
\]
By proper base change and \cite[Proposition 3.3 (iii)]{phd_paper}, we have
\begin{eqnarray*}
  \omega^*i^*j_*j^*M & \simeq & \omega^1i^*q_* \widetilde{\jmath}_* j^{*}M \\
  & \simeq & \omega^1q_{Z*}\omega^1 \widetilde{\imath}^*\widetilde{\jmath}_* j^{*}M.
\end{eqnarray*}
The functor $\omega^1\pi_{Z*}$ is t-non-positive \cite[Proposition 4.14]{phd_paper}. So it is enough to prove Lemma \ref{lemma:omega_chi} in the case $S$ is integral.

We want to improve the geometric situation using alterations. 

\begin{lemma}\label{lemma:alt_res}
There exists a projective alteration $\pi:S'\ra S$ such that
\begin{itemize}
\item $S'$ is regular, and 
\item $C\times_{S} S'$ extends to a projective semi-stable curve $\bar{f}':\bar{C}'\ra S'$ with geometrically connected fibers such that $\bar{C}'$ is regular and $\sigma'=\sigma\times_{S} S'$ extends to a section $\bar{\sigma}'$ of $\bar{C}'/S'$, which lands in the smooth locus.
\end{itemize}
\end{lemma}
\begin{proof}
  The pair $(C,\sigma)$ determines a $U$-point of the stack $\overline{\CM}_{g,1}$ of genus $g$ stable curves with a section. By a standard argument using the existence and properness of the moduli stack $\overline{\CM}_{g,1}$, there exists a projective alteration $\pi_1:S_1\ra S$ (with $S_1$ again integral) such that, if we write $U_1=U\times_S S_1$, $C_1=C\times_S S_1$ and so on, the pair $(C_1,\sigma_1)$ extends to a point $(\bar{C}_1,\bar{\sigma}_1)\in\overline{\CM}_{g,1}$. Note that, by definition of $\overline{\CM}_{g,1}$, such a curve still has geometrically connected fibers. In particular, $\bar{\sigma}_1$ factors through the $S_1$-smooth locus $\bar{C}_1^{\sm}$ of $\bar{C}_1$. By \cite[Lemma 5.7]{DeJong_families}, by doing a further projective alteration of $S_1$, we can also assume that $\bar{C}_1$ is quasi-split over $S_1$ in the sense of loc. cit., i.e., that on every fiber the singular points and the tangents to the singular points are rational.

  The closed subset $f_1(\mathrm{Sing}(\bar{C}_1))$ is a proper closed subset since $\bar{C}_1$ is generically smooth over a generically regular scheme $S_1$. By resolution of singularities by alterations applied to the pair $(S_1,f_1(\mathrm{Sing}(\bar{C}_1)))$, which is possible by hypothesis on $S$, there exists a projective alteration $\pi_2:S_2\ra S_1$ with $S_2$ integral and regular, and with $D:=\pi_2^{-1}(f_1(\mathrm{Sing}(\bar{C}_1)))$ a strict normal crossings divisor. Put $U_2=U\times_S S_2$ and so on. Then $(\bar{C}_2,\bar{\sigma}_2)\in \overline{\CM}_{g,1}(S_2)$ is still a quasi-split stable curve, which is moreover smooth outside of the strict normal crossings divisor $D$.

By \cite[Proposition 5.11]{DeJong_families}, there exists a projective modification $\phi_3:\bar{C}_3\ra \bar{C}_2$ which is an isomorphism outside of $\mathrm{Sing}(\bar{C}_2)$ (so in particular over $\bar{C}^{\sm}_2$, and over $S_2\setminus D$ via $\bar{f}_2$), and such that $\bar{C}_3$ is regular and the composite $\bar{f}_3=\bar{f}_2\circ \phi_3:\bar{C}_3\ra S_2$ is a projective semi-stable curve. Since $\phi_3$ is an isomorphism on $\bar{C}^{\sm}_2$, the section $\bar{\sigma}_2$ lifts to a section $\bar{\sigma}_3:S_2\ra \bar{C}_3$ of $\bar{f}_3$.

We now put $S'=S_2$, $\bar{C}'=\bar{C}_3$, $\bar{f}'=\bar{f}_3$, and $\bar{\sigma}'=\bar{\sigma}_3$. These satisfy all the requirements of the conclusion of the lemma.
\end{proof}

Let us fix an alteration $\pi$ as in Lemma \ref{lemma:alt_res}. Consider the following diagram of schemes with cartesian squares.
\[
\xymatrix{
  V' \ar[d]_{\pi_V} \ar[r]_{j'} & S' \ar[d]_{\pi} & Z' \ar[l]^{i'} \ar[d]_{\pi_Z} \\
  V \ar[r]_{j} & S & Z \ar[l]^{i}
}
\]
By construction, the morphism $\pi_V$ is quasi-\'etale. By \cite[Lemme 2.1.165]{Ayoub_thesis}, we have that $j^*M$ is a direct factor of $\pi_{V*}\pi_V^*j^*M$. We are thus reduced to show that $\omega^*i^*j_*\pi_{V*}\pi_V^*j^*M$ is t-non-positive. By proper base change and \cite[Proposition 3.3 (iii)]{phd_paper}, we have
\begin{eqnarray*}
  \omega^*i^*j_*\pi_{V*}\pi_V^*j^*M & \simeq & \omega^1i^*\pi_* j'_* \pi_V^*\Sigma^\infty \CM(-1) \\
  & \simeq & \omega^1\pi_{Z*}\omega^1 i'^*j'_* \Sigma^\infty \pi_V^*\CM(-1).
\end{eqnarray*}

The functor $\omega^1\pi_{Z*}$ is t-non-positive \cite[Proposition 4.14]{phd_paper}, and $\pi_V^*\CM(-1)$ is a Deligne $1$-motive. Hence we can assume that $S'=S$ and that $C$ itself has an extension $\bar{f}:\bar{C}\ra S$ satisfying the conclusions of Lemma~\ref{lemma:alt_res}.

By the distinguished triangles associated to the weight filtration of $\CM$, we can treat each piece separately and assume that $\CM$ is pure. As usual, the abelian variety case is the most complicated, and we give the argument only for that case.

\noindent
\underline{Degeneration of the Jacobian and conclusion of the proof of Lemma \ref{lemma:omega_chi}:}

We assume $\CM=\Jac(C/V)$. Let us fix some notation in the following diagram with cartesian squares.

\[
\xymatrix{
  C \ar[r]_{\tilde{\jmath}} \ar[d]_{f} & \bar{C} \ar[d]_{\bar{f}} & \partial C \ar[d]_{\partial f} \ar[l]^{\tilde{\imath}}\\
  U \ar[r]_{j} & S & Z \ar[l]^{i}
}
\]
By \cite[Corollary 3.20]{phd_paper}, the motive $\Sigma^\infty \Jac(C/U)(-1)[-1]$ is a direct factor of $f_*\BQ_C[+1]$, via the section $\sigma$ (we use the fact that $C/U$ has geometrically connected fibers). More precisely, using the adjunction $(\sigma^*,\sigma_*)$ and $f\sigma=\id$, we get a map
\[
\pi_0:f_*\BQ_C\ra f_*\sigma_*\sigma^* \BQ_C\simeq \BQ_U.
\]
From \cite[Corollary 3.20]{phd_paper}, we deduce that $\Sigma^\infty \Jac(C/U)(-1)[-2]$ is a direct factor of $\mathrm{Fib}(\pi_0)$, hence that $\Sigma^\infty \Jac(C/U)(-1)[-1]$ is a direct factor of $\mathrm{Fib}(\pi_0[+1])$. So it is enough to show that that $\omega^1 i^*j_* \mathrm{Fib}(\pi_0[+1])\simeq \mathrm{Fib}(\omega^1 i^*j_*\pi_0)[+1]$ is t-non-positive.

By base change and \cite[Proposition 3.3 (iii)]{phd_paper}, we have $\omega^1i^*j_* f_*\BQ_C\simeq \omega^1 (\partial f)_*\omega^1\tilde{\imath}^*\tilde{\jmath}_*\BQ_{C}$. Moreover, if we write $\partial\pi_0:\partial f_*\BQ_{\partial C}\ra \BQ_Z$ obtained from $\partial \sigma$ in the same fashion as $\pi_0$ was obtained from $\sigma$, the following square
\[
\xymatrix{
\omega^1 (\partial f)_* \ar[r] \ar[d]_{\omega^1 \partial\pi_0} & \omega^1 i^* j_* f_*\BQ[+1] \ar[d]_{\omega^1i^*j_*\pi_0} \\
\BQ[+1] \ar[r] & \omega^1 i^* j_*\BQ[+1]
}
\]
commutes. Applying localisation, we then can complete it into the following diagram with distinguished rows.
\[
\xymatrix{
  \omega^1 (\partial f)_* \BQ[+1]\ar[r] \ar[d]_{\omega^1 \partial\pi_0} & \omega^1 i^* j_* f_*\BQ[+1] \ar[d]_{\omega^1i^*j_*\pi_0} \ar[r]
 &\omega^1 \partial f_* \omega^1 \tilde{\imath}^!\BQ[+2] \ar[d]  \ar[r]^{+} &\\
\BQ[+1] \ar[r] & \omega^1 i^* j_*\BQ[+1] \ar[r] & \omega^1 i^!\BQ[+2] \ar[r]^{+} & .
}
\]
By Lemma~\ref{lemma:shriek_purity} below, the rightmost terms of both triangles are both t-non-positive. We deduce that the central terms lie in $\DA^1(Z)_{\leq 1}$, and that to show that $\omega^1 i^*j_* \mathrm{Fib}(\pi_0)$ is non-positive, it is enough to prove that the morphism
\[
H_1(\omega^1 i^* j_* f_*\BQ[+1]) \ra H_1(\omega^1 i^*j_*\BQ[+1])
\]
is injective. From the triangle and Lemma~\ref{lemma:shriek_purity} below, we see that
\[
H_0(\omega^1\partial f_*\BQ)\simeq H_1(\omega^1\partial f_*\BQ[+1])\simeq H_1(\omega^1 i^* j_* f_*\BQ[+1])
\]
and
\[
\BQ\simeq H_1(\BQ[+1]) \simeq H_1(\omega^1 i^*j_*\BQ[+1]).
\]
So we have to prove that the morphism $H_0(\omega^1 \partial f_*\BQ[+1])\ra \BQ$ induced by $\partial \sigma$ is injective. We prove that it is in fact an isomorphism. First of all, since $\partial f$ is a proper curve, we have $\partial f_*\BQ\in \DA^1_c(Z)$, so that $\omega^1\partial f_*\BQ \simeq \partial f_*\BQ$. By \cite[Proposition 3.24]{Ayoub_Etale}, it is enough to show that, for all points $z\in Z$, the morphism $z^*H_0(\partial f_*\BQ_{\partial C}) \ra \BQ_z$ is an isomorphism. We have $\dim(Z)<\dim(S)$, so by the induction hypothesis and Lemma~\ref{lemma:str_pull}, the pullback functor $z^*$ t-exact for the $1$-motivic t-structures. We thus have $z^*H_0(\partial f_*\BQ)\simeq H_0(z^*\partial f_*\BQ)\simeq H_0(\partial (f_z)_*\BQ)$, with $\partial f_z:\bar{C}_z\ra z$. The curve $\bar{C}_z$ is geometrically connected, hence the morphism $H_0(\partial f_{z*})\ra \BQ$ induced by the point $\bar{\sigma}_z\in \bar{C}_z(z)$ is an isomorphism. This concludes the proof.
\end{proof}

The following lemma was used in the proof of Theorem \ref{main_theorem}.

\begin{lemma}\label{lemma:shriek_purity}
  Let $S$ be a regular scheme, and $i:Z\ra S$ be a closed immersion with $Z$ reduced and nowhere dense. Let $k_0:Z_0\ra Z$ be an open immersion with $Z_0$ regular of pure codimension $1$ in $S$, containing all points of $Z$ which are of codimension $1$ in $S$ (i.e., all generic points of irreducible components of $Z$ which are divisors); note that $Z_0$ can be empty. Let $M\in \MM^{1,\sm}_c(S)$ be a smooth constructible $1$-motive. Then
  \[
\omega^1i^!M\simeq \omega^1(k_{0*}k_0^*i^*M(-1))[-2]\simeq \omega^0(k_{0*}k_0^*i^*M)(-1)[-2].
\]
In particular, $\omega^1i^!\BQ[+2]$ is $t^1_{\MM}$-non-positive.
\end{lemma}

\begin{proof}
  Extend $Z_0$ into a stratification $\emptyset=\bar{Z}_{m+1}\subset \bar{Z}_{m}\subset\ldots\subset \bar{Z}_0=Z$ by closed subsets with $Z_i:=\bar{Z}_i\setminus \bar{Z_{i-1}}$ regular and equidimensional. Write $c_i$ for the codimension of $Z_i$ in $S$; by hypothesis, we can arrange that $c_i\geq 2$ for $i>0$. Write $Z_i\stackrel{k_i}{\ra} \bar{Z}_i\stackrel{l_i}{\leftarrow} \bar{Z}_{i+1}$. By localisation, we have distinguished triangles
\[
l_{0*}l_0^! i^!M \ra i^!M \ra k_{0*} k_0^! i^!M  \rap
\]
\[
l_{1*}l_1^!l_0^! i^!M \ra l_0^!i^! \ra k_{1*}k_1^! l_0^!M\rap
\]
\[
\dots
\]
\[
l_{m*}l_m^!l_{m-1}^!\dots i^!M l_{d-1}^!\dots i^!M \ra k_{m*}k_m^!l_{d-1}^!\dots i^!M
\]
For all $i$, since $S$ and $Z_i$ are regular, the immersion $k_i= i\circ l_1\circ\dots l_{i-1} :Z_i\ra S$ is a regular immersion. By absolute purity in the form of \cite[Proposition 1.7]{phd_paper} for the smooth motive $M$, we deduce that $k_i^!M\simeq k_i^*M(-c_i)[-2c_i]$. Combining this formula with the triangles above and the fact that $\omega^1M(-2)=0$ for any $M\in \DA^{\coh}$ implies that $\omega^1 i^!M\simeq \omega^1(k_{0*}k_0^*i^*M(-1)[-2])$.

We have $\omega^1 (k_{0*}k_0^*i^*M(-1))\simeq \omega^0(k_{0*}k_0^*i^*M\BQ)(-1)$ by \cite[Corollary 3.9 (iv)]{phd_paper}.

We have $\BQ(-1)\in\MM^1(Z_0)$. By \cite[Proposition 4.14 (v)]{phd_paper}, the motive $\omega^1k_*(\BQ(-1))$ is then $t^1$-negative. This concludes the proof.\end{proof}

\bibliographystyle{amsplain}
\bibliography{constructible}

\providecommand{\bysame}{\leavevmode\hbox to3em{\hrulefill}\thinspace}
\providecommand{\MR}{\relax\ifhmode\unskip\space\fi MR }
\providecommand{\MRhref}[2]{%
  \href{http://www.ams.org/mathscinet-getitem?mr=#1}{#2}
}
\providecommand{\href}[2]{#2}
\begin{thebibliography}{10}

\bibitem{Ayoub_thesis}
Joseph Ayoub, \emph{Les six op\'{e}rations de {G}rothendieck et le formalisme
  des cycles \'evanescents dans le monde motivique. {I \& II}}, Ast\'erisque
  (2007), no.~314 \& 315.

\bibitem{Ayoub_Etale}
\bysame, \emph{La r\'ealisation \'etale et les op\'erations de {G}rothendieck},
  Ann. Sci. \'Ec. Norm. Sup\'er. (4) \textbf{47} (2014), no.~1, 1--145.
  \MR{3205601}

\bibitem{bvk}
Luca Barbieri-Viale and Bruno Kahn, \emph{On the derived category of
  1-motives}, Ast\'erisque (2016), no.~381, xi+254. \MR{3545132}

\bibitem{BBD}
A.~A. Be{\u\i}linson, J.~Bernstein, and P.~Deligne, \emph{Faisceaux pervers},
  Analysis and topology on singular spaces, {I} ({L}uminy, 1981), Ast\'erisque,
  vol. 100, Soc. Math. France, Paris, 1982, pp.~5--171. \MR{751966 (86g:32015)}

\bibitem{Brion_iso}
Michel Brion, \emph{Commutative algebraic groups up to isogeny}, Doc. Math.
  \textbf{22} (2017), 679--725. \MR{3650225}

\bibitem{DeJong_families}
A.~Johan de~Jong, \emph{Families of curves and alterations}, Ann. Inst. Fourier
  (Grenoble) \textbf{47} (1997), no.~2, 599--621. \MR{1450427}

\bibitem{SGA3_1}
Philippe Gille and Patrick Polo (eds.), \emph{Sch{\'e}mas en groupes ({SGA} 3).
  {T}ome {I}. {P}ropri\'et\'es g\'en\'erales des sch\'emas en groupes},
  Documents Math\'ematiques (Paris) [Mathematical Documents (Paris)], 7,
  Soci\'et\'e Math\'ematique de France, Paris, 2011, S{\'e}minaire de
  G{\'e}om{\'e}trie Alg{\'e}brique du Bois Marie 1962--64. [Algebraic Geometry
  Seminar of Bois Marie 1962--64], A seminar directed by M. Demazure and A.
  Grothendieck with the collaboration of M. Artin, J.-E. Bertin, P. Gabriel, M.
  Raynaud and J-P. Serre, Revised and annotated edition of the 1970 French
  original. \MR{2867621}

\bibitem{Morel}
Sophie Morel, \emph{Cohomologie d'intersection des vari\'et\'es modulaires de
  {S}iegel}, Compos. Math. \textbf{147} (2011), no.~6, 1671--1740. \MR{2862060}

\bibitem{Neeman_book}
Amnon Neeman, \emph{Triangulated categories}, Annals of Mathematics Studies,
  vol. 148, Princeton University Press, Princeton, NJ, 2001. \MR{1812507
  (2001k:18010)}

\bibitem{orgogozo}
Fabrice Orgogozo, \emph{Isomotifs de dimension inf\'erieure ou \'egale \`a un},
  Manuscripta Math. \textbf{115} (2004), no.~3, 339--360. \MR{2102056
  (2005j:14022)}

\bibitem{phd_paper}
Simon Pepin~Lehalleur, \emph{Triangulated categories of relative 1-motives},
  ArXiv e-prints (2015), Preprint.

\bibitem{Temkin_distillation}
Michael Temkin, \emph{Tame distillation and desingularization by
  {$p$}-alterations}, Ann. of Math. (2) \textbf{186} (2017), no.~1, 97--126.
  \MR{3665001}

\bibitem{vaish_gluing}
V.~{Vaish}, \emph{Punctual gluing of $t$-structures and weight structures},
  ArXiv e-prints (2017), Preprint.

\end{thebibliography}

\medskip \medskip

\noindent{Freie Universit\"{a}t Berlin, Arnimallee 3, 14195 Berlin, Germany} 

\medskip \noindent{\texttt{simon.pepin.lehalleur@gmail.com}}

\end{document}